\documentclass[12pt,reqno]{article}
mm \textheight=8.9in

\usepackage{amssymb}
\usepackage{amsmath}
\usepackage{amsthm}
\usepackage{pb-diagram}
\usepackage{color}

\newcommand{\br}[3]{{$#1$}$\lower4pt\hbox{$\tp\atop\raise4pt \hbox{$\scriptscriptstyle{#2}$}$} ${$#3$}}
\newcommand{\tw}[3]{{$#1$}${\,\scriptscriptstyle {#2}}\atop\raise9pt\hbox{$\scriptstyle\tp$} ${$#3$}}
\newcommand{\ttps}[2]{{#1}\raise5pt\hbox{$\lower12pt\hbox{$\scriptstyle\tp$}\atop \lower0pt\hbox{$\tilde\;$}$}\raise4.5pt\hbox{${\scriptstyle{#2}}$}}
\newcommand{\st}[1]{\mbox{${\,\scriptscriptstyle {#1}}\atop\raise5.5pt\hbox{$*$}$}}

\newcommand{\rd}[1]{\mbox{${\,\scriptscriptstyle {#1}}\atop\raise5.5pt\hbox{$\bullet$}$}}
\newcommand{\rt}[1]{\otimes_\chi}
\newcommand{\lt}[1]{\mbox{${\,\scriptscriptstyle {#1}}\atop\raise5.5pt\hbox{$\ltimes$}$}}
\newcommand{\btr}{\raise1.2pt\hbox{$\scriptstyle\blacktriangleright$}\hspace{2pt}}
\newcommand{\btl}{\raise1.2pt\hbox{$\scriptstyle\blacktriangleleft$}\hspace{2pt}}

\newcommand{\lcr}{\raise1.0pt \hbox{${\scriptstyle\rightharpoonup}$}}
\newcommand{\rcr}{\raise1.0pt \hbox{${\scriptstyle\leftharpoonup}$}}

\newcommand{\ttp}{{\lower12pt\hbox{$\tp$}\atop \hbox{$\tilde\;$}}}

\newcommand{\id}{\mathrm{id}}

\newcommand{\Bc}{\mathcal{B}}

\newcommand{\Ha}{{H}}
\newcommand{\A}{{A}}
\newcommand{\B}{{B}}

\newcommand{\Mc}{\mathcal{M}}

\newcommand{\C}{\mathbb{C}}
\newcommand{\Z}{\mathbb{Z}}
\newcommand{\R}{\mathbb{R}}

\newcommand{\tp}{\otimes}

\newcommand{\zt}{\zeta}

\newcommand{\U}{U}

\newcommand{\F}{F}
\newcommand{\ve}{\varepsilon}
\newcommand{\gm}{\gamma}

\newcommand{\la}{\lambda}

\newcommand{\End}{\mathrm{End}}

\newcommand{\Span}{\mathrm{Span}}

\newcommand{\Hom}{\mathrm{Hom}}

\newcommand{\Ind}{\mathrm{Ind}}

\newcommand{\g}{\mathfrak{g}}
\renewcommand{\b}{\mathfrak{b}}

\newcommand{\h}{\mathfrak{h}}

\newcommand{\mb}{\boldsymbol{m}}
\newcommand{\xb}{\boldsymbol{x}}
\newcommand{\yb}{\boldsymbol{y}}
\newcommand{\zb}{\boldsymbol{z}}
\newcommand{\wb}{\boldsymbol{w}}

\newcommand{\ztb}{\boldsymbol{\zeta}}
\newcommand{\omb}{\boldsymbol{\omega}}
\newcommand{\om}{\omega}

\newcommand{\s}{\mathfrak{s}}
\newcommand{\n}{\mathfrak{n}}

\newcommand{\nn}{\nonumber}

\newcommand{\p}{\mathfrak{p}}
\renewcommand{\l}{\mathfrak{l}}

\newcommand{\al}{\alpha}

\newcommand{\bt}{\beta}

\newcommand{\be}{\begin{eqnarray}}
\newcommand{\ee}{\end{eqnarray}}

\newtheorem{thm}{Theorem}[section]
\newtheorem{propn}[thm]{Proposition}
\newtheorem{lemma}[thm]{Lemma}
\newtheorem{corollary}[thm]{Corollary}

\theoremstyle{definition}

\newcount\prg

\newcommand{\parag}{\advance\prg by1 {\noindent\bf\thesection.\the\prg\hspace{6pt}}}

\begin{document}
\title{On dynamical adjoint functor}
\author{A.~Mudrov \\
\small Department of Mathematics, \\
\small University of Leicester,
\\
\small University Road
Leicester LE1 7RH\\
\small e-mail: am405@le.ac.uk\\
}
\date{ }

\maketitle
\begin{abstract}
We give an explicit formula relating the dynamical adjoint functor and dynamical twist over
nonalbelian base to the invariant pairing on parabolic Verma modules.
As an illustration, we give explicit $\U\bigl(\s\l(n)\bigr)$- and  $\U_\hbar\bigl(\s\l(n)\bigr)$-invariant star product on projective spaces.
\end{abstract}
\underline{\small AMS classification codes}: 16D90, 17B10, 17B37, 53D55.\\
\section{Introduction}
In this paper, we  clarify certain points arising in the theory of dynamical Yang-Baxter equation (DYBE), namely, we give an explicit expression of
the dynamical adjoint functor through the invariant pairing on parabolic Verma modules (PVM).

The DYBE appeared as a generalization of the ordinary Yang-Baxter equation in the mathematical physics literature
\cite{GN,AF,F} and
was actively studied in the 90-s,  see e.g. \cite{ES1,EV3}. It was later realized that it is related to quantization
of homogeneous Poisson-Lie manifolds, \cite{DM}. Namely, the star-product can be obtained by a reduction of  dynamical twist, which is a
left-invariant differential tri-operator on the group. In return, this finding has extended the framework of the DYBE, which had been originally formulated
over a commutative cocommutative base Hopf algebra, \cite{ES1}, to a general nonabelian base.

A recipe of constructing dynamical twist was suggested in \cite{DM}, through dynamical adjoint functor
 (DAF) between certain module categories
associated with a Levi subalgebra $H$ in a reductive (classical or quantum) universal enveloping algebra $\U$. One of them is a certain subcategory of finite dimensional $H$-modules, while the
other is the parabolic $O$-category, both equipped with tensor multiplication by finite dimensional $\U$-modules.

Remark that quantization of function algebras on homogeneous space involve only scalar PVM. General PVM appear
in quantization of associated vector bundles as projective modules over function algebras, as discussed
in \cite{DM}. This is where the nonabelian base really plays a role, along with the corresponding
dynamical twist and DAF.

An alternative approach to quantization was employed in \cite{AL}, where
the star product on semisimple coadjoint orbits of simple complex Lie groups was constructed directly from the  Shapovalov form on scalar PVM.  It was clear that the methods of \cite{DM} and \cite{AL}  were  close and based on similar underlying ideas.
Relation of the Shapovalov form on Verma modules with the dynamical twist was already indicated in earlier works on DYBE in the special case of Cartan base, \cite{EV3}.
 This construction had motivated the generalization for the nonabelian base, which was given in \cite{DM}, however, without straight use of the Shapovalov form.
A sort of "nonabelian paring" associated with the triangular factorization of (quantized) universal enveloping algebras, which is equivalent to Shapovalov form in representations,  was employed in \cite{EEM} for construction of  the  dynamical twist. It was done directly,  bypassing DAF. Thus, the explicit relation of the Shapovalov form to DAF  over general Levi subalgebra,  which is a more fundamental object than dynamical twist, has not been given much attention in the literature. In the present work we do it in a most elementary way.

We would like to  mention the following two papers in connection with the present work. In \cite{EE}, the dynamical twist is constructed with the use of the ABRR equation, \cite{ABRR}. The DAF
is also present there, but with no explicit connection with the Shapovalov form. Another paper of interest,
 \cite{KST}, directly generalizes the ideas of \cite{AL}. Remarkably,  the approach of  \cite{KST} can be
 suitable for certain conjugacy classes with non-Levi isotropy subgroups, which drop from the framework of
the  DYBE in its present version, but still can be quantized in a similar way, \cite{M}.

As an illustration, we give the star product on the
 homogeneous space $GL(n+1)/GL(n)\times GL(1)$ that is equviariant under the action of either classical or quantum group $GL(n+1)$.  In this simple case the Shapovalov form can be calculated explicitly. We show that its  $\U(\g)$-invariant limit coincides with
the star product on the projective space obtained in \cite{BBEW} by a different approach.

Unfortunately, the journal version contains a few typos and mistakes, which are corrected here. The erratum is added after the list of references.

\section{Parabolic Verma modules}
Let $\g$ be a simple complex Lie algebra.
Fix a Cartan subalgebra and denote by  $\b^\pm\subset \g$  the  Borel subalgebras relative to $\h$. Consider a Levi subalgebra $\l\subset \g$ containg
$\h$ and denote by $\p^\pm= \l+\b^\pm$ the corresponding parabolic subalgebras.
The nil-radicals $\n^\pm\subset \p^\pm$ complement $\l$ in the triangular
decomposition $\n^-\oplus \l\oplus \n^+$.

Denote by $\U_\hbar(\g)$, $\U_\hbar(\l)$, and  $\U_\hbar(\p^\pm)$ the quantum universal
enveloping algebras of the total Lie algebra $\g$, the Levi subalgebra $\l$, and
the parabolic subalgebras $\p^\pm$.  They are Hopf $\C[[\hbar]]$-algebras, with the inclusions
$\U_\hbar(\l) \subset \U_\hbar(\p^\pm)\subset \U_\hbar(\g)$ being Hopf algebra
homomorphisms. There are subalgebras $\U_\hbar(\n^\pm)\subset \U_\hbar(\p^\pm)$,
which are deformations of the classical universal enveloping algebras $\U(\n^\pm)$ and
which facilitate the triangular factorization
\be
\label{tr_fac}
\U_\hbar(\g)=\U_\hbar(\n^-) \U_\hbar(\l)\U_\hbar(\n^+).
\ee
This factorization makes $\U_\hbar(\g)$ a free $\U_\hbar(\n^-)-\U_\hbar(\n^+)$-bimodule generated by $\U_\hbar(\l)$. Accordingly, the parabolic subalgebras admit free factorizations
\be
\label{tr_fac1}
\U_\hbar(\p^-)=\U_\hbar(\n^-)\U_\hbar(\l),
\quad
\U_\hbar(\p^+)=\U_\hbar(\l)\U_\hbar(\n^+),
\ee
giving rise to  (\ref{tr_fac}).
Factorization (\ref{tr_fac1}) has the structure of smash product, as $\U_\hbar(\n^\pm)$ are invariant under the adjoint action of $\U_\hbar(\l)$ on $\U_\hbar(\g)$.
 Note that $\U_\hbar(\n^\pm)$ are neither Hopf algebras nor coideals
over $\U_\hbar(\l)$.

Let $A$ be a finite dimensional $\U_\hbar(\l)$-module (finite and free over $\C[[\hbar]]$). It becomes a $\U_\hbar(\p^\pm)$-module with the
 trivial action of $\U_\hbar(\n^\pm)$.
The parabolic Verma modules $M_A^\pm$ over $\U_\hbar(\g)$ are defined as induced from $A$:
$$
M_A^\pm=\U_\hbar(\g)\tp_{\U_\hbar(\p^\pm)} A \simeq\U_\hbar(\n^\mp) A.
$$
The last isomorphism indicates that $M_A^\pm$ are free $\U_\hbar(\n^\mp)$-modules generated by $A$.
Consider $M_A^\pm$ as a module over $\U_\hbar(\p^\mp)$ by restriction.
In what follows, we use sections of the action map $\U_\hbar(\p^\mp)\tp A\to M_A^{\pm}$.
We call the image of $M_A^{\pm}$ in $\U_\hbar(\p^\mp)\tp A$ under such a section  {\em a lift} of $M_A^{\pm}$.
The presentation $M_A^\pm\simeq \U_\hbar(\n^\mp) A$ is an example of
lift.

For any finite dimensional  $\U_\hbar(\l)$-module $A$ let $A^*$ denote the (left) dual to $A$.
The dual action
of $\U_\hbar(\l)$ on $A^*$ is given by  $(h\al)(a)=(\al)\bigl(\gm(h)a\bigr)$,
for $\al \in A^*$,  $a\in A$  and $h\in \U_\hbar(\l)$.
The triangular factorization of $\U_\hbar (\g)$ relative to $\U_\hbar (\l)$ defines
a projection $\U_\hbar (\g)\to \U_\hbar (\l)$, $u\mapsto [u]_\l$, which is $\U_\hbar (\l)$-invariant with respect
to the left and right regular action.
 The pairing
$$
\U_\hbar (\g)\tp A^*\tp \U_\hbar (\g)\tp A\to \C[[\hbar]], \quad \langle u\tp \al, v\tp a\rangle =\langle \al, [\gm(u) v]_\l a\rangle
$$
is $\U_\hbar (\g)$-invariant by construction and factors through the pairing between  $M^-_{A^*}$ and $M^+_{A}$.
It  is  non-degenerate if
and only if the modules $M^-_{A^*}$ and $M^+_{A}$ are irreducible. This pairing is $\U_\hbar(\g)$-invariant and  equivalent to the
contravariant Shapovalov form on $M^+_{A}$, \cite{Jan2}.

In the sequel, we need the following fact about induced modules of Hopf algebras.
\begin{lemma}
Suppose $\U$ is a Hopf algebra and $H$ is a Hopf subalgebra in $\U$.
Let $A$ be an $H$-module and $V$ be a $\U$-module regarded as an $H$ module by
restriction.
Then there are natural isomorphisms
$$
\Ind_H^\U A\tp  V\simeq \Ind_H^\U (A\tp  V),
\quad
V\tp \Ind_H^\U A\simeq \Ind_H^\U (V\tp  A),
$$
where $\Ind_H^\U $ designates  induction from an $H$-module to a $\U$-module.
\label{ind-tens}
\end{lemma}
The proof is a straightforward use of Hopf algebra yoga involving coproducts, antipode and counit. The isomorphisms are identical on $A\tp V$ and $V\tp A$, respectively.

\begin{propn}
\label{factorization}
The tensor product $M^+_A\tp M^-_B$ is isomorphic to the induced module $\U_\hbar(\g)\tp _{\U_\hbar(\l)} (A\tp B)$.
\end{propn}
\begin{proof}
Based on Lemma \ref{ind-tens},
\be
M_A^+\otimes M_B^-&\simeq&\Ind_{\U_\hbar(\p^+)}^{\U_\hbar(\g)}(A\otimes M_B^-)\simeq\Ind_{\U_\hbar(\p^+)}^{\U_\hbar(\g)}\bigl(A\otimes \Ind_{\U_\hbar(\l)}^{\U_\hbar(\p^+)}(B)\bigr)
\nn\\
&\simeq&\Ind_{\U_\hbar(\p^+)}^{\U_\hbar(\g)}\Ind_{\U_\hbar(\l)}^{\U_\hbar(\p^+)}(A\otimes B)\simeq
\Ind_{\U_\hbar(\l)}^{\U_\hbar(\g)}(A\otimes B),
\ee
as required.
\end{proof}
\section{Dynamical adjoint functor}
Let us recall the definition of dynamical twist over general base \cite{DM}.
For simplicity, we take for the base a Hopf algebra $H$ assuming it to be a Hopf subalgebra in the total
Hopf algebra $\U$. A dynamical twist is an invertible element $F\in H\tp \U\tp \U$ subject to the cocycle identity
$$
(\id_\Ha \tp \id_\U \tp \Delta)(\F)
(\delta\tp \id_\U \tp \id_\U)(\F)
=
(\id_\Ha \tp \Delta \tp \id_\U)(\F)(\F\tp 1_\U)
$$
and the normalization condition $(\id\tp \ve\tp \id)(F)=1\tp 1\tp 1 =(\id\tp \id\tp \ve)(F)$.
In practice, the base Hopf algebra $H$ needs to be replaced by a certain "localization", which is already not a coalgebra
but only a right coideal, see \cite{EE,EEM}.

In representation-theoretical terms, the dynamical twist is a family of operators
$$
\F_{A,V,W}\colon (A\tp V)\tp W\to A\tp (V\tp W)
,
$$
where $V,W$ are $\U$-modules  and  $A$ is an $H$-module.
The above mentioned localization of $H$ means that not all $A$ are admissible. The cocycle identity turns into
$$
\F_{A,Z,V\tp W}\F_{A\tp Z,V,W}
=
\F_{A,Z\tp V,W}(\F_{A,Z, V}\tp \id_W).
$$
Here $Z$ is another $\U$-module, and the second factor in the left-hand side regards $A\tp Z$ as an $H$-module through
the coaction.

We recall a general construction of DAF. This functor was introduced in \cite{DM}, where it was used for construction of the dynamical twist.
The name "dynamical adjoint" should be taken as a single term, as the functor of concern is rather "dynamizaion" of
the adjoint functor to the restriction functor than
dualization of anything.

Suppose $\Mc$ is a monoidal category and $\Bc$, $\Bc'$ are two (right) module categories over $\Mc$. A functor $J$ from $\Bc$ to $\Bc'$ is called dynamical adjoint if
for all $A, B\in \Bc$ and $V\in \Mc$ there is an isomorphism
$$
\Theta\colon\Hom_{\Bc}(B, A\tp V)\simeq \Hom_{\Bc'}(J(B), J(A)\tp V).
$$
Here $\tp$ stands for the actions of $\Mc$ on $\B$ and $\B'$.
Given such functor, the dynamical twist is a morphism
$F\colon A\tp V\tp W\to A\tp V\tp W$ whose $J$-image is the composition
$$
J({A\tp V\tp W})\stackrel{\Theta(\id_{A\tp V\tp W})}{\longrightarrow} J({A\tp V})\tp W\stackrel{\Theta(\id_{A\tp V})\tp \id_W}{\longrightarrow} J({A})\tp V\tp W.
$$

In our special case, $\Mc$ is the category of finite dimensional representations
of $\U_\hbar(\g)$, $\Bc$ is a certain subcategory of finite dimensional representations of  $\U_\hbar(\l)$, and  $\Bc'$ is the category of integrable modules over $\U_\hbar(\g)$. The functor $J$ is the parabolic induction, so $J(A)=M^+_A$ on objects. The category $\Bc$ is determined by the requirements that $M_A^+$ is irreducible once $A$ is irreducible and $\Bc$  is invariant under  tensoring with objects from $\Mc$ regarded as  $\U_\hbar(\l)$-modules.

That the functor is DAF follows from Proposition
\ref{factorization}.
Indeed, if the module $M_A^+$ is irreducible, then $M_{A^*}^-$ is its (restricted) dual, and
$$
\Hom_{\g}(M^+_B, M^+_A\tp  V)\simeq
\Hom_{\g}(M^-_{A^*}\tp M^+_B,  V)\simeq
\Hom_{\l}(A^*\tp B,V)\simeq \Hom_{\l}(B,A\tp V),
$$
as required. The isomorphism in the middle  is the Frobenius reciprocity facilitated by Proposition \ref{factorization}.

Assuming $M^+_A$ irreducible, denote by $S_{A,A^*}\in M^+_A\tp M_{A^*}^-$ the $\U_\hbar(\g)$-invariant
canonical element of the pairing between $M_{A^*}^-$ and $M^+_A$.
For the dual  bases $\{x_i\} \subset M^+_A$, $\{x^i\} \subset M_{A^*}^-$,
it equals $S_{A,A^*}=\sum_i x_i \tp x^i$.
We will use the following Sweedler-like notation for the presentation
$$
S_{A,A^*}= S_1\tp S_2= S_-S_A\tp S_+S_{A^*}.
$$
where  $S_-\tp S_A\tp S_+\tp S_{A^*}\in \U_\hbar(\p^-)\tp A\tp \U_\hbar(\p^+)\tp A^*$
 symbolizes  a lift of $S_{A,A^*}$.

We have the following obvious property of the family $\{S_{A,A^*}\}$.
Suppose $B$ is an $\U_\hbar(\l)$-submodule in $A$. It is separable
as a direct summand, so the dual module $B^*$ is separable as a direct summand
in $A^*$. Let $\pi\colon A\to B$ and $\psi\colon A^*\to B^*$ be the intertwining projections.
Let $\hat\pi \colon M_{A}^+\to M_{B}^+$ and $\hat\psi \colon M_{A^*}^-\to M_{B^*}^-$ be the induced   $\U_\hbar(\g)$-morphisms.
Then the projection
$(\hat\pi\tp \hat\psi)(S_{A,A^*})=S_-\pi(S_B)\tp S_+\psi(S_{B^*})$
 is equal to the canonical element $S_{B,B^*}$.
Thus, a lift of $S_{B,B^*}$ can be obtained from a lift of $S_{A,A^*}$ in a natural way.

Every $\U_\hbar(\g)$-module $V$ is at the same time
a $\U_\hbar(\l)$-module by restriction. It becomes an $\U_\hbar(\p^\pm)$-module when extended trivially to  $\U_\hbar(\n^\pm)$. This representation
of $\U_\hbar(\p^\pm)$ is different as compared to
restricted from $\U_\hbar(\g)$. To distinguish  it
from the representation $\rho$ restricted from $\U_\hbar(\g)$, we
use the notation $\rho_{\pm}$.

Fix a $\U_\hbar(\g)$-module $V$ and assign to every element of $M_{A^*}^-$ a linear operator $A\tp V\to V$ by
\be
\label{M-map}
u \al \colon a\tp v \mapsto \bigl(\al\tp (\rho( u^{(1)})\bigr)\Bigl(\rho_{+}\bigl(\gm (u^{(2)})\bigr) ( a\tp v)\Bigr),
\ee
where $u\tp \al \in \U_\hbar(\p^+)\tp A^*$ is a lift of $u \al$, and $v\in V$. The map is correctly defined. Indeed, regarding  $M_{A^*}^-$ as a
$\U_\hbar(\p^+)$-module, it is isomorphic
to $\mathrm{Ind}^{\p^+}_{\l}A^*$.
The assignment (\ref{M-map}) coincides on $(ux) \al$ and $u (x\al)$ for every $x$ from $\U_\hbar(\l)$,
because $\rho_{+}\bigl(\gm (x^{(2)})\bigr)=\rho\bigl(\gm (x^{(2)})\bigr)$.
One can check that the constructed map $M_{A^*}^-\to \End(A\tp V, V)$ is $\U_\hbar(\l)$-invariant.
Note that for the classical universal enveloping algebras formula (\ref{M-map}) simplifies to
$$
u \al \colon a\tp v \mapsto \al(a) \rho( u)(v),
$$
because the lift $u$ can be taken  in the Hopf subalgebra $\U(\n^+)\subset \U(\g)$, which  is annihilated by $\rho_{+}$.

DAF gives rise to the collection of intertwining operators
$$M_{A\tp V}^+\stackrel{\Theta(\id_{A\tp V})}{\longrightarrow} M_{A}^+\tp V.$$
\begin{propn}
\label{J-S}
The restriction to $A\tp V\subset M^+_{A\tp V}$ of the intertwining operator
$
\Theta(\id_{A\tp V})
$
acts by the assignment
$$
\Theta(\id_{A\tp V})\colon a\tp v\mapsto
S_1\tp  S_2(a\tp v),
$$
where $S_{A,A^*}=S_1\tp S_2\in M^+_A\tp M^-_{A^*}$
\end{propn}
\begin{proof}
The operator $\Theta(\id_{A\tp V})$ factorizes into the composition
$$
M_{A\tp V}^+\to M_{A}^+\tp M_{A^*}^-\tp M_{A\tp V}^+\simeq M_{A}^+\tp \Ind^\g_\l(A^*\tp A\tp V)
\to M_{A}^+\tp V,
$$
where the left arrow is the coevaluation  $1\mapsto S_{A,A^*}$, while right arrow is the induced extension from
the evaluation map $A^*\tp A\tp V\to V$. Note that the isomorphism in the
middle is identical on $A^*\tp A\tp V$.
Applying this composition to an element $a\tp v\in A\tp V\subset M_{A\tp V}^+$ gives the result immediately.
\end{proof}
Remark that for the classical universal enveloping algebras we can write the simple formula
$$
\Theta(\id_{A\tp V})\colon a\tp v\mapsto
S_1\tp  S_{A^*}(a)S_+(v).
$$
Here the factor $S_+$ acts on $v\in V$ as an element of $\U(\g)$.

To relate the dynamical twist to the invariant pairing, consider the projection
$P\colon M_A^+\to A$ defined as the composition
$$M_A^+\to A\tp A^* \tp  M_A^+\to A,$$
where the left arrow is induced by coevaluation $\C[[\hbar]]\to A\tp A^*$, and the right map is
the invariant pairing between $A^*\subset M_{A^*}^-$ and $M_A^+$.
By construction, $P$ is $\U_\hbar(\l)$-invariant and identical on $A\subset M_A^+$.
Moreover, one can check that
it is $\U_\hbar(\p^-)$-invariant.

The operator $P$ implements the isomorphism
\be
\label{dynF}
\Hom_\g(M_B^+,M_A^+\tp V)\to \Hom_\l(B,A\tp V)
\ee
that is inverse to  $\Theta$. Namely, given a $\U_\hbar(\g)$-invariant operator $M_B^+ \to M_A^+\tp V$,
restrict it to $B$ and compose with $P\tp \id_V$, to get a $\U_\hbar(\l)$-operator $B\to \A\tp V$.
Therefore the dynamical twist $F$ factorizes to the chain
$$
A\tp V\tp W\hookrightarrow
M_{A\tp V\tp W}\to M_{A\tp V}\tp W \to M_{A}\tp V\tp W \to A    \tp V\tp W.
$$
\begin{propn}
The dynamical twist $F_{A,V,W}$ is expressed through the canonical element
$S_{A\tp V,V^*\tp A^*}=S_-S_{A\tp V}\tp S_+S_{V^*\tp A^*}\in M_{A\tp V}^+\tp M_{V^*\tp A^*}^-$ by
the formula
\be
\label{F-S}
F(a\tp v\tp w)=(\rho_{-}\tp \rho)\circ\Delta(S_-)(S_{A\tp V})\tp (S_+S_{V^*\tp A^*})(a\tp v\tp w).
\ee
\end{propn}
\begin{proof}
Immediate consequence of (\ref{dynF}) and Proposition \ref{J-S}.
\end{proof}

Remark that for the classical universal enveloping algebras
the natural lift gives $S_\pm\in \U(\n^\pm)$ killed by $\rho_\pm$.
The formula (\ref{F-S}) takes the simple form
$$
F(a\tp v\tp w)=(\id_A\tp S_-)(S_{A\tp V})\tp  S_+(w)S_{V^*\tp A^*}(a\tp v),
$$
where $S_\pm$ act on $V$, $W$ as $\U(\g)$-modules. In the quantum case,
the subalgebra $\U_\hbar(\n^-)$ can be taken a left coideal, and the formula can also be simplified by replacing
$(\rho_{-}\tp \rho)\circ\Delta(S_-)$ with $(\id\tp \rho)(S_-)$.

Now suppose that $A=\mathbb{C}_\lambda$ is a scalar $\U_\hbar(\l)$-module corresponding to weight $\la$.
Denote by $S^\la=\la(S_-^{(1)})S_-^{(2)}\tp S_+^{(1)}\la^*(S_+^{(2)})$,
where $S_-\tp S_+$ is a lift of  $S_{\C^\la, \C^{\la^*}}\in M_{\la}^+\tp M_{\la^*}^-$ to $\U_\hbar(\p^-)\tp \U_\hbar(\p^+)$.
\begin{propn}
The operators  $F^{\la}$ and $S^\la$ coincide on
invariants $V^\l\tp W^\l\subset V\tp W$:
$$
F^\la(v\tp w)=S^\la(v\tp  w),\quad v\in V^\l, \quad w\in W^\l.
$$
\end{propn}
\begin{proof}
Apply the projections $\C_\la\tp V\to \C_\la\tp V^\l$, $\C_\la\tp W\to \C_\la\tp W^\l$ to (\ref{F-S}).
\end{proof}
The formula for  $S^\la$ can be simplified to $S^\la=S_-\tp S_+$ by an appropriate choice of basis.
\section{Quantum algebra $\U_\hbar(\s\l(n))$}
Further we apply the above theory to construct the star-product on the homogeneous space
space $GL(n+1)/GL(n)\times GL(1)$, for which case the invariant pairing on PVM can be explicitly calculated. We will focus on the situation of quantum groups,
because the classical case can be readily obtained from that by a  certain limit procedure.

Recall the definition of the quantized universal enveloping algebra
$\U_\hbar(\s\l(n))$, see \cite{Dr1}.
Let $R$ and  $R^+$ denote respectively the root system and the set of positive roots of the Lie algebra $\s\l(n)$.
The set $\Pi_+=(\al_1,\al_1,\ldots, \al_{n-1})$ of simple positive roots is equipped with the natural ordering.

The quantum group $\U_q(\s\l(n))$ is generated by
$
e_{\al}, f_{\al}, h_\al$, $\al\in \Pi_+
$,
subject to the relations
$$
[h_{\al},e_{\bt}]= (\al,\bt) e_{\al},
\quad
[h_{\al},f_{\bt}]=- (\al,\bt) e_{\al},
\quad
[e_{\al},f_{\bt}]=\delta_{\al,\bt} \frac{q^{h_{\al}}-q^{-h_{\al}}}{q-q^{-1}},
$$
where $(.,.)$ is the inner product on  $\h^*=\Span(R)$. Here $q=e^{\hbar}$ with
$\hbar$ being the deformation parameter.

Also, the Chevalley generators $e_\al$, $f_\al$ satisfy the Serre relations
$$
e_{\al}^{2}
e_{\bt}
-
(q+q^{-1})
e_{\al}e_{\bt}e_{\al}
+
e_{\bt}e_{\al}^{2}
=0,\quad
f_{\al}^{2}
f_{\bt}
-
(q+q^{-1})
f_{\al}f_{\bt}f_{\al}
+
f_{\bt}f_{\al}^{2}
=0.
$$
if $(\al,\bt)=-1$ and
$
[e_\al,e_\bt]=0=[f_\al,f_\bt]
$
if $(\al,\bt)=0$.

The comultiplication $\Delta$ and antipode $\gm$ are defined on the generators by
$$
\Delta(h_\al)=h_\al\tp 1+1\tp h_\al, \quad \gm(h_\al)=-h_\al,
$$
$$
\Delta(e_\al)=e_\al\tp 1+q^{h_\al}\tp e_\al, \quad \gm(e_\al)=-q^{-h_\al}e_\al,
$$
$$
\Delta(f_\al)=f_\al\tp q^{-h_\al}+1\tp f_\al, \quad \gm(f_\al)=-f_\al q^{h_\al}.
$$
The counit homomorphism $\ve$ annihilates $e_\al$, $f_\al$, $h_\al$.

The elements $e_\al, h_\al$ generate the positive Borel subalgebra  $\U_\hbar (\b^+)$
 in $\U_q(\s\l(n))$.
Similarly, $f_\al, h_\al$ generate the negative Borel subalgebra  $\U_\hbar (\b^-)$.
They are deformations of the classical Borel subalgebras whose Poincar\'e-Birgoff-Witt
basis is generated by the Cartan generators constructed as follows.

Every positive root has the form $\mu=\al_k+\al_{k+1}+\ldots+\al_m$, for some $m>k$;
the integer $m-k+1$ is called height of $\mu$.
Put
$$
e_\mu=[e_{k},[e_{k+1},\ldots [e_{m-1},e_{m}]_q\ldots]_q]_q
,\quad
\tilde e_\mu=q^{2(m-k)}[e_{k},[e_{k+1},\ldots [e_{m-1},e_{m}]_{\bar q}\ldots]_{\bar q}]_{\bar q},
$$
where $e_k=e_{\al_k}$.
The roots can be written in an orthogonal basis $\{\ve_i\}_{i=1}^n$ as $\ve_i-\ve_j$, $i,j=1,\ldots,n$, $i\not=j$.
Lexicographically ordered pairs $(i,j)$ induce
an ordering  on positive roots $\ve_i-\ve_j$, $i<j$,
consistent with the ordered the basis $(\al_1,\al_2,\ldots, \al_{n-1})\subset \h^*$.
The ordered monomials in $e_\mu, \mu\in R^+$, form a BPW basis in $\U_\hbar (\b^+)$
over the Cartan subalgebra  in $\U_\hbar (\h)$ generated by $h_\al$.

As all $\al_i\in \Pi_+$ enter positive roots  at most once, we may regard elements of
$R^+$ as sets of simple roots.
This makes sense of writing $\mu \subset \nu$ and $\nu - \mu$ for some pairs of roots
$\mu, \nu\in R^+$.  Also, $\al\cap\bt$ determines a root unless it is empty.
\begin{lemma}
For all positive roots $\mu$, the vectors $e_\mu$ and $\tilde e_\mu$
are related by the antipode,
$
\gm(\tilde e_\mu)=-q^{-h_\mu} e_\mu.
$
\label{tilde-antipod}
\end{lemma}
\begin{proof}
If $\al<\bt$ are adjacent simple roots, then the following calculation
$$
\gm(q^2[e_\al, e_\bt]_{\bar q})=q^2[q^{-h_\bt}e_\bt, q^{-h_\al} e_\al]_{\bar q}
=q^{-h_{\al+\bt}}q[e_\bt, e_\al]_{\bar q}
=-q^{-h_{\al+\bt}}[e_\al, e_\bt]_q
$$
proves the statement for roots of height $2$. The general case is processed by
induction on the height of the root.
\end{proof}
In the classical limit $\mod \hbar$, the elements $e_\mu$ turn into positive
root vectors of the Borel subalgebra $\b^+\subset \g$ and  modulo $\hbar$ they coincide
with  $\tilde e_\mu$. For our purposes, we need both
$e_\mu$ and $\tilde e_\mu$.

The commutation relations among the root vectors $e_\mu$ are described below.
\begin{propn}
\label{prop_rel1}
Let  $\mu$ and $\nu$ be positive roots such that $\mu<\nu$.
\begin{enumerate}
\item Suppose that $\mu+\nu \in R^+$. Then
\be
[e_\mu,e_\nu]_{q}&=&e_{\mu+\nu}.
\label{complex_root_Serre}
\ee
\item Suppose that   $\mu+\nu \not \in R^+$. Then
\be
\quad [e_\nu,e_{\mu}]_{\bar q}&=&0, \quad \mu\cap \nu=\nu, \quad \nu<\mu-\nu\in R^+,
\label{boundary1}
\\[1pt]
[e_{\mu}, e_\nu]_{\bar q}&=&0
, \quad \mu\cap \nu=\nu, \quad \nu>\mu-\nu \in R^+,
\label{boundary2}
\\[1pt]
[e_\mu,e_\nu]&=&0,
\quad
\mu\cap \nu=\nu, \quad \mu-\nu \not \in R^+
\quad \mbox{or } \nu\cap \mu=\emptyset,
\label{not boundary0}
\\[1pt]
[e_{\mu},e_{\nu}]&=&-(q-q^{-1})e_{\mu\cup \nu}e_{\mu\cap \nu}
\nn\\&=&
-(q-q^{-1})e_{\mu\cap \nu}e_{\mu\cup \nu},\quad
\quad
\nu\not=\mu\cap \nu\in \R^+.
\label{not boundary}
\ee
\end{enumerate}
\end{propn}
\begin{proof}
Formula (\ref{complex_root_Serre}) is just the definition of $e_{\mu+\nu}$ if the height of $\mu$ is $1$.
If $\mu=\al_i+\ldots +\al_{k-1}$ with $k>i$, put $\mu'=\al_i+\ldots +\al_{k-1}$; then
Then
$$
[e_\mu,e_\nu]_{q}=[[e_\mu',e_{k-1}]_{q},e_\nu]_{q}=[e_\mu',[e_{k-1},e_\nu]_{q}]_{q}=e_{\mu+\nu}.
$$
by induction on the height of $\mu$.

Commutation relations (\ref{boundary1}) and (\ref{boundary2})
 are generalizations of the Serre relations and  follow from  Lemma \ref{Great Auxiliary} by induction on heights
of $\mu$ and $\nu$.
The case  $\mu+\nu,\mu-\nu\not \in R^+$ falls  either in situation
(\ref{not boundary0}) or (\ref{not boundary}).
Let us check (\ref{not boundary0}). The case $\mu\cap \nu=\emptyset$ is clear. The alternative is $\mu=\nu'+\nu+\nu''$, where $\nu',\nu''\in R^+$ and $\nu'<\nu'<\nu''$.
Then (\ref{not boundary0}) follows from  (\ref{xyyz}), where
we put $x=e_{\nu'}$, $y=e_{\nu}$,  $z=e_{\nu''}$.

The second equality in (\ref{not boundary}) follows from (\ref{not boundary0}).
To prove the first equality, put,
under the assumption of (\ref{not boundary}),
$
\mu=\mu'+\mu\cap \nu
$
for some $\mu',\nu''\in R^+$ such that $\mu'<\mu\cap \nu < \nu ''$.
To check the formula  (\ref{not boundary}), pull the root vectors
$e_{\mu'}$ and $e_{\mu\cap \nu}$ to the right in
$
(e_{\mu'}e_{\mu\cap \nu}-qe_{\mu\cap \nu}e_{\mu'})
e_{\nu}
$
using already proved (\ref{complex_root_Serre}), (\ref{boundary1}), (\ref{boundary2}),
and  (\ref{not boundary0}). This gives  (\ref{not boundary}).
\end{proof}

The elements $\tilde e_{\mu}$ satisfy similar identities, which can be readily derived from these by applying the antipode anti-isomorphism.

Root vectors of the negative Borel subalgebra are defined as follows.
Let $\omega$ be the involutive automorphism of $\U_\hbar(\s\l(n))$ defined on the simple root vectors by
$e_\al \leftrightarrow -f_\al, h\to -h$.
Introduce $f_\mu$ as $-(-q)^{k-1}\omega(e_\mu)$, where $k$ is the height of $\mu$.
Explicitly, for $\mu=\al_k+\al_{k+1}+\ldots+\al_m$, $k<m$, put
$$
f_\mu=[f_{m},[f_{m-1},\ldots [f_{k+1},f_{k}]_{\bar q}\ldots]_{\bar q}]_{\bar q},
$$
where $f_i=f_{\al_i}$.
The elements $f_\mu$, $\mu\in R^+$ generate a PBW basis of $\U_\hbar(\n^-)$, which can be obtained
by the isomorphism $\omega\colon \U_\hbar(\n^+)\to \U_\hbar(\n^-)$.
The commutation relations on $f_\mu$ can be derived from  Proposition \ref{prop_rel1} by applying the automorphism $\omega$.

Further we need more commutation relations among the elements of $\U_\hbar(\s\l(n))$.

\begin{propn}
\label{prop_rel2}
For every positive roots $\gm$
$$
[e_{\gm},f_{\gm}]=\frac{q^{h_\gm}-q^{-h_\gm}}{q-q^{-1}}.
$$
If  $\mu<\gm$ and   $\mu+\gm$ is a root, then
$$
[e_{\gm},f_{\gm+\mu}]=-q^{-1}f_{\mu}q^{-h_\gm},
\quad
[e_{\gm},f_{\mu+\gm}]=f_{\mu}q^{h_\gm},
\quad
[f_{\gm},e_{\gm+\mu}]=f_{\mu}q^{h_\gm},
\quad
[f_{\gm},e_{\mu+\gm}]=-q e_{\mu}q^{-h_\gm},
$$
$$
[f_{\gm},e_{\mu+\gm+\nu}]=0,
\quad
[e_{\gm},f_{\mu+\gm+\nu}]=0,
$$
$$
[e_{\mu+\gm},f_{\gm+\nu}]
=(q-q^{-1}) f_{\nu}e_\mu q^{-h_\gm}.
$$
\end{propn}
The proof of these formulas is given in Appendix.
\begin{corollary}
\label{prop_rel3}
For every positive roots $\nu$ and $\mu$ such that $\nu+\mu\in R^+$, and positive integer $k$,
\be
[e_{\mu},f^k_{\mu+\nu}]&=&-q^{-1}\frac{q^{2k}-1}{q^{2}-1}f_{\mu+\nu}^{k-1}f_{\nu}q^{-h_\mu},
\nn
\\[1pt]
[e_{\nu},f_{\mu+\nu}^k]&=&
q^{-k+1}\frac{q^{2k}-1}{q^{2}-1}f_{\mu}f_{\mu+\nu}^{k-1}q^{h_\nu}.\nn
\\[1pt]
[e_{\nu},f^k_{\nu}]&=&
f^{k-1}_{\nu}\Bigl( q^{h_\nu+1}\frac{1-q^{-2k}}{(q-q^{-1})^2}+q^{-h_\nu-1}\frac{1-q^{2k}}{(q-q^{-1})^2}\Bigr)
,\nn
\ee
\end{corollary}

\section{The invariant form in simple case}
In this section we calculate the invariant form on scalar PVM
assuming $\l=\g\l(n)\oplus \g\l(1)\subset \g\l(n+1)=\g$.
The positive root system of $\l$ is generated by the simple roots
$\al_2,\ldots \al_{n}$. The nil-radical Lie algebras $\n^\pm$ are spanned by the root vectors corresponding to
the roots $\pm\al_1,\pm(\al_1+\al_2),\ldots,\pm (\al_1+\ldots+\al_n)$.

Introduce generators
$$x_1=e_{\al_1},\quad x_2=e_{\al_1+\al_2}, \quad \ldots,  \quad x_n=e_{\al_1+\ldots + \al_n},$$
$$\tilde x_1=\tilde e_{\al_1},\quad \tilde x_2=\tilde e_{\al_1+\al_2}, \quad \ldots,  \quad \tilde  x_n=\tilde e_{\al_1+\ldots + \al_n},$$
$$y_1=f_{\al_1},\quad y_2=f_{\al_1+\al_2}, \quad \ldots,  \quad y_n=f_{\al_1+\ldots +\al_n}.$$
The subalgebras $\U_\hbar(\n^+)$ and $\U_\hbar(\n^-)$ are generated, respectively, by $\{x_i\}$, and $\{y_i\}$.

Let $v_\la$ be the generator of the PVM,  $M^+_\la$, induced from a character $\la\in \h^*$ of the Levi subalgebra $\l$.
The weight $\lambda$ is proportional
to the basis weight $\ve_1$ in the standard orthogonal basis of the  $\g\l(n+1)$ weight space.
With abuse of notation, we will use the same symbol for the coefficient
$\la \sim  \la\ve_1$. This should not cause any confusion in the context.

The generator of the dual module $M^-_{-\la}$ will be denoted by $v_{-\la}$.
The monomials
$$
y_n^{m_n}y_{n-1}^{m_{n-1}}...y_1^{m_1}v_{\la}\in M^+_\la, \quad
\tilde x_n^{m_n}\tilde x_{n-1}^{m_{n-1}}...\tilde x_1^{m_1}v_{\la}\in M^-_{-\la},
$$
where  $m_i$ are non-negative integers, form a basis in $M_\la^+$ and, respectively, in $M_{-\la}^-$.

\begin{lemma}
The matrix coefficient
$\langle \tilde x_n^{k_n}\tilde x_{n-1}^{k_{n-1}}\ldots
\tilde x_1^{k_1}v_{-\la},y_n^{m_n}y_{n-1}^{m_{n-1}}\ldots y_1^{m_1}v_{\la}\rangle$
vanishes unless $k_i=m_i$, for all $i=1,\ldots, n$.
\end{lemma}
\begin{proof}
Since $\gm(\tilde x_i)=q^{-h_\mu}x_\mu$, this matrix coefficient is proportional
to
$$
\langle v_{-\la}, x_1^{k_1}\ldots
 x_{n-1}^{k_{n-1}}x_n^{k_n}y_n^{m_n}y_{n-1}^{m_{n-1}}\ldots y_1^{m_1}v_{\la}
 \rangle .
 $$
Suppose first that $k_n>m_n$. Commutation of $x_n$ with $y_n^{m_n}$ reduces
the degree $m_n$ by one and produces a factor from $\U_\hbar(\h)$.
Pushing it further to the right we get a $e_{\al_n}$-factor by commutation with $y_{n-1}$.
This factor commutes with all elements $y_i$, $i=1,\ldots n-1$, and can be placed to the
rightmost position, where it annihilates $v_\la$. Similar effect will be produced  by commutation of  $x_n$ with other
$y_i$. Thus, only the term from commutation with $y_n^{m_n}$ survives on the way of $x_n$ to the right.
Repeating this for other $x_n$-factors, we see that the matrix coefficient vanishes if $k_n>m_n$.
If $k_n<m_n$, similar arguments can be used when pushing $y_n$ to the left till they meet $v_{-\la}$.

Thus, the only possibility for the matrix coefficient to survive is $k_n=m_n$.
In this case, commutation of $x_n^{m_n}$ with $y_n^{m_n}$ produces an element from $\U_\hbar(\h)$, which
in its turn gives rise to a scalar factor. This reduces the consideration to the case to $m_n=k_n=0$,
and one can repeat the above reasoning. The obvious induction on $n$ completes the proof.
\end{proof}
Further we  calculate the matrix coefficients explicitly.
Let $\Z_+$ denote set of non-negative integers.
Fix an $n$-tuple $\mb=(m_1,\ldots    , m_n)\in \Z_+^n$ and
put $|\mb|=\sum_{i=1}^{n}m_i$. It is convenient to
pass to "quantum integers" $\hat k= q^{k-1}[k]_q=\sum_{i=0}^{k-1}q^{2i}$,
for non-negative $k$.

\begin{propn}
The non-vanishing matrix coefficients of the invariant paring are given by
\be
\langle \tilde x_n^{m_n}\tilde x_{n-1}^{m_{n-1}}\ldots
\tilde x_1^{m_1}v_{-\la},y_n^{m_n}y_{n-1}^{m_{n-1}}\ldots y_1^{m_1}v_{\la}\rangle
=(-1)^{|\mb|}
q^{-\psi(\mb)}\prod_{i=1}^n\hat m_i!\prod_{j=0}^{|\mb|-1}[\la-j]_q.
\label{mat_coef}
\ee
where
$
\psi(\mb)=(\la+\frac{1}{2}) |\mb|-\frac{1}{2}|\mb|^2
$.
\end{propn}
\begin{proof}
Applying Lemma \ref{tilde-antipod} we get for the matrix coefficient
$$
(-1)^{\sum_{i=1}^{n}m_i}
\langle v_{-\la}, (q^{-h_1} x_1)^{m_1} \ldots  (q^{-h_{n}}x_n)^{m_n}y_n^{m_n}\ldots y_1^{m_1}v_{\la}
\rangle =$$
$$
=(-1)^{\sum_{i=1}^{n}m_i}
q^{-\la\sum_{i=1}^{n}m_i+\frac{1}{2}(\sum_{i=1}^{n}m_i)(\sum_{i=1}^{n}m_i-1)
+\frac{1}{2}\sum_{i=1}^{n}m_i(m_i-1)}\times
$$
$$
\times
\langle v_{-\la}, x_1^{m_1} \ldots  x_n^{m_n}y_n^{m_n}\ldots y_1^{m_1}v_{\la}
\rangle .$$
The numeric coefficient is equal to
$(-1)^{|\mb|}q^{-\la |\mb|-|\mb|+\frac{1}{2}|\mb|^2
+\frac{1}{2}\sum_{i=1}^n m_i^2}$.
For every non-negative integer $m$ and a complex parameter $z$ denote
$$
Z_q[m,z]=\frac{1}{q-q^{-1}}(q^{z}\frac{1-q^{-2m}}{1-q^{-2}}-q^{-z}\frac{1-q^{2m}}{1-q^{2}})
=[m]_q[z-m+1]_q.
$$

The matrix coefficient
$
\langle v_{-\la}, x_1^{m_1} \ldots x_n^{m_n}y_n^{m_n}\ldots y_1^{m_1}v_{\la}\rangle
$
is found  to be
$$
\frac{1}{q-q^{-1}}(q^{\la-\sum_{i=1}^{n-1}m_i}\frac{1-q^{-2 m_n}}{1-q^{-2}}-
q^{-\la+\sum_{i=1}^{n-1}m_i}\frac{1-q^{2 m_n}}{1-q^{2}})
\times \ldots=Z_q[m_n,\la-\sum_{l=1}^{n-1}m_i]\times\ldots
$$
In the product omitted on the right, we go down  from
$m_n$ to $1$ in the first argument of $Z_q$. Then repeat the procedure as though the dimension of the vector
$\mb$ were $n-1$ rather then $n$, and proceed until we get to $n=0$.
The result for the matrix coefficient will be
$$
\prod_{i=1}^n\prod_{j=1}^{m_i}
Z_q[j,\la-\sum_{l=1}^{i-1}m_i]
=\prod_{i=1}^n[m_i]_q!
\prod_{j=0}^{|\mb|-1}[\la-j]_q.
$$
To complete the proof, one should pass to the q-integers $\hat m_i$.
\end{proof}
Let $\C_\hbar[G]$ denote the affine coordinate ring on the quantum group $GL_q(n+1)$ (the Hopf
dual to the quantized universal enveloping algebra $\U_\hbar\bigl(\g\l(n+1)\bigr)$.)
Denote by $\C_\hbar[G]^\l$ the subalgebra of $\U_\hbar(\l)$-invariants in
$\C_\hbar[G]$ under the left co-regular action. This space naturally inherits the right co-regular action of $\U_\hbar(\g)$
compatible
with the multiplication  $\cdot_\hbar$. It is known that $\cdot_\hbar$ is a star product, \cite{T}.

Notice that the tensors $y_i\tp \tilde x_i$ commute with each other,
as follows from Proposition \ref{prop_rel1}. For the reasons that
will be clear later, we would like to modify $y_i$ to $\tilde y_i$
in such a way that the tensors $D_i=\tilde y_i\tp  \tilde x_i$, $i=1,\ldots,n$ satisfy
the quantum plane relations
\be
D_j D_i= q^{2} D_i D_j, \quad j<i.
\label{q-plane}
\ee Another condition on this transformation
is to leave the matrix coefficients of the invariant  paring untouched.
This can be achieved by means of the replacement
\be
y\mapsto \tilde y_i=q^{-(\eta_i,\la)}y_iq^{-\eta_i}, \quad i=1,\ldots, n,
\label{tilde_y}
\ee
where $\eta_i\in \h$ are to be determined.
\begin{propn}
There exists a unique sequence $\eta_i\in \h$, $i=1,\ldots,n$, such that
$D_i=\tilde y_i\tp  \tilde x_i$ satisfy the quantum plane relations
and
\be
\langle \tilde x_n^{m_n}\tilde x_{n-1}^{m_{n-1}}\ldots
\tilde x_1^{m_1}v_{-\la},\tilde y_n^{m_n}\tilde y_{n-1}^{m_{n-1}}\ldots \tilde y_1^{m_1}v_{\la}\rangle
=(-1)^{|\mb|}
q^{-\psi(\mb)}\prod_{i=1}^n\hat m_i!\prod_{j=0}^{|\mb|-1}[\la-j]_q.
\ee
All other matrix coefficients are zero.
\label{mat_coeff1}
\end{propn}
\begin{proof}
Introduce a new basis $\{\bt_i\}_{i=1}^n$ in $\h$ setting $\bt_i=\al_1+\ldots +\al_i$.
Note that the vectors $\tilde x_i$, $\tilde y_i$ carry weights $\pm\bt_i$.
The Gram matrix $(\bt_i,\bt_j)$ and its inverse are, respectively
$$
\left(
\begin{array}{rrrrrrr}
2 & 1&1&\ldots &1\\
1 & 2&1& \ldots &1\\
\ldots
\\
1 & 1&1& \ldots &2
\end{array}
\right)
,\quad
\left(
\begin{array}{rrrrrrr}
\frac{n}{n+1} & \frac{-1}{n+1}&\frac{-1}{n+1}&\ldots &\frac{-1}{n+1}\\
\frac{-1}{n+1} & \frac{n}{n+1}&\frac{-1}{n+1}& \ldots &\frac{-1}{n+1}\\
\ldots
\\
\frac{-1}{n+1} & \frac{-1}{n+1}&\frac{-1}{n+1}& \ldots &\frac{n}{n+1}
\end{array}
\right).
$$
Define $\eta_i=\sum_{k=1}^n B_{ik}\bt_k$ through the system of equations
\be
(\eta_j,\bt_i )=- 2+(\eta_i,\bt_j ),
\quad j<i
,\quad
 \quad (\eta_i,\bt_j)=0, \quad i\geqslant j.
 \label{cartan_factors}
\ee
The transition matrix is uniquely defined and equal to
$$
B
=
\left(\begin{array}{rrrrrrr}
0 & -2&-2&\ldots &-2\\
0 & 0&-2& \ldots &-2\\
&&\ldots
\\
0 & 0&0& \ldots &0
\end{array}
\right)
\left(
\begin{array}{rrrrrrr}
\frac{n}{n+1} & \frac{-1}{n+1}&\frac{-1}{n+1}&\ldots &\frac{-1}{n+1}\\
\frac{-1}{n+1} & \frac{n}{n+1}&\frac{-1}{n+1}& \ldots &\frac{-1}{n+1}\\
\ldots
\\
\frac{-1}{n+1} & \frac{-1}{n+1}&\frac{-1}{n+1}& \ldots &\frac{n}{n+1}
\end{array}
\right).
$$
Now we can complete the proof. Notice that the left group of equations (\ref{cartan_factors})
facilitates the quantum plane relations (\ref{q-plane}).
Further, it is clear that the matrix coefficients in the new basis involving $\tilde y_i$ are
proportional to the old ones. For the non-vanishing matrix coefficient (\ref{mat_coeff1}) is equal to
$$
q^{\sum_{k=1}^n\frac{m_k(m_k-1)}{2}(\eta_k,\bt_k)+\sum_{i>j}(\eta_i,\bt_j)m_im_j}
\langle \tilde x_n^{m_n}\tilde x_{n-1}^{m_{n-1}}\ldots
\tilde x_1^{m_1}v_{-\la}, y_n^{m_n}y_{n-1}^{m_{n-1}}\ldots y_1^{m_1}v_{\la}\rangle.
$$
The scalar multiplier disappears due to the left condition in (\ref{cartan_factors}). This completes the proof.
\end{proof}
\section{Star product on complex projective spaces}
In this section we apply the results of the preceding considerations to construction of $\U_\hbar\bigl(\s\l(n+1)\bigr)$-invariant star product on the homogeneous space
$GL(n+1)/GL(n)\times GL(1)$.
We start with the following well known fact.
\begin{lemma}
If  $D_1,\ldots, D_n$ satisfy the quantum plain relations (\ref{q-plane}), then
\be
(D_1+\ldots +D_n)^m=\sum_{m_1+\ldots + m_n=m} \frac{\hat m!}{\hat m_1!\ldots \hat m_n!}
D_n^{m_n}\ldots D_1^{m_1},
\label{q-binomial}
\ee
for all non-negative integers $m$.
\end{lemma}
This lemma can be easily proved by induction on $n$.
Now we can construct the star product.
Define the tensor $\tilde \yb\tp \tilde \xb=\sum_{i=1}^n \tilde y_i\tp \tilde x_i$.
\begin{thm}
The element
\be
\label{mult}
\sum_{m=0}^\infty
\frac{(-1)^{m}q^{(\la+1)m-m^2}}{[m]_q!
\prod_{j=0}^{m-1}[\la-j]_q}
(\tilde \yb \tp \tilde \xb)^m,
\ee
is a lift of the inverse invariant form on $M_{-\la}\tp M_{\la}$.
\end{thm}
\begin{proof}
An immediate consequence of Proposition \ref{mat_coeff1} and Lemma \ref{q-binomial}.
\end{proof}
The  operator (\ref{mult}) is not quasiclassical modulo $\hbar$. To make it a star product deformation
of the ordinary multiplication in $\C[G]^\l$, we need to extend the ring of scalars by
Laurent series and consider the module $M^+_{\frac{\la}{\hbar}}$.
\begin{corollary}
For $f,g\in \C_\hbar[G]^\l$, the multiplication
$$f*_\hbar g:=\sum_{\mb \in \Z_+^n}
\frac{(-1)^{|\mb|}q^{(\la+\frac{1}{2})m-\frac{1}{2}m^2}}{\prod_{i=1}^n\hat m_i!
\prod_{j=0}^{|\mb|-1}[\frac{\la}{\hbar }-j]_q}
\bigl(\tilde y_n^{m_n}\ldots
\tilde y_1^{m_1} f\bigr)\cdot_\hbar\bigl(\tilde x_n^{m_n}\ldots \tilde x_1^{m_1}g\bigr),
$$
is a $\U_\hbar\bigl(\g\l(n+1)\bigr)$-invariant star product on $\C_\hbar[G]^\l$
 under the right
co-regular  action.
\end{corollary}
Remark that the star product given by this formula involves the star product in $\C_\hbar[G]$,
whose explicit expression through the classical multiplication in $\C[G]^\l$
 is unknown. Therefore it cannot be regarded as perfectly explicit.

Let us reserve the same notation  for the classical limits of the root vectors $x_i$, $y_i$.
Recall that $x_i$ and $\tilde x_i$ have the same classical limits.
\begin{corollary}
For $f,g\in \C[G]^\l$, the multiplication
\be
f*_t g:=\sum_{\mb \in \Z_+^n}
\frac{(-t)^{|\mb|}}{\prod_{i=1}^nm_i!
\prod_{j=0}^{|\mb|-1}(\la-jt)}
\bigl( x_n^{m_n} \ldots
 x_1^{m_1} f\bigr)\cdot\bigl(y_n^{m_n}\ldots y_1^{m_1}g\bigr),
\label{star_P^n}
\ee
is a $\U\bigl(\g\l(n+1)\bigr)$-invariant star-product on $\C[G]^\l$
 under the right
co-regular  action.
\end{corollary}
This multiplication is obtained from (\ref{mult}) in two steps: taking limit $\hbar \to 0$
and subsequent replacement of $\la$ by $\frac{\la}{t}$.

In classical universal enveloping algebra setting
the (scalar reduced) dynamical twist $F^{\la}$ takes the form
\be
F^\la=\sum_{m=0}^\infty \frac{(-t)^{m}}{m!
\prod_{j=0}^{m-1}(\la-jt)}
(\xb\tp \yb)^m.
\label{star_P^n_1}
\ee
The tensor $\xb\tp \yb=\sum_{i=1}^n x_i\tp y_i$ is the $\U_\hbar(\l)$-invariant
element of $\n^+_\l\tp \n^-_\l$.
\section{Comparison with earlier results}
In the present section we compare the star-product on $GL(n+1)/GL(n)\times GL(1)$
with that on the complexified projective space $\C P^{n}$ regarded as a real
manifold. This star product was obtained in \cite{BBEW} by completely different methods.
In both cases they form a one parameter family. In our setting, it corresponds to the
highest weight of the module $M_\la$, while in \cite{BBEW} to the radius of
 $\C P^{n}$.
We prove that the two star products  coincide up to a shift of this parameter.

Let us rewrite the star product (\ref{star_P^n})  in local coordinates.
Introduce a parametrization of a neighborhood of the identity in $GL(n+1)$:
$$
(\ztb,\omb,h)\mapsto
\left(
\begin{array}{ccc}
1&0 \\
\omb&1_n
\end{array}
\right)
\left(
\begin{array}{ccc}
1&\ztb \\
0&1_n
\end{array}
\right)
\left(
\begin{array}{ccc}
a&0 \\
0&A
\end{array}
\right),
$$
where $\ztb$ and $\omb$ are, respectively, $n$-dimensional row and column,
$1_n\in GL(n)$ is the unit matrix, and
$h=\left(
\begin{array}{ccc}
a&0 \\
0&A
\end{array}
\right)\in H=GL(1)\times GL(n).
$
\newcommand{\fb}{\boldsymbol{f}}
\newcommand{\phib}{\boldsymbol{\phi}}
A function $\varphi\in \C[G]$ is $H$-invariant if and only if it is independent of $h$ in these local
coordinates, $\varphi(\ztb,\omb,h)=\varphi(\ztb,\omb,1_{n+1})$. Then we shall write simply $\varphi(\ztb,\omb)$.

It is obvious that the left-invariant vector field $x_i$ is represented by  the partial derivative $\frac{\partial}{\partial  {\zt_i}}$.
To evaluate the left-invariant vector field $y_i$ at the point $(\ztb,\omb,1)$,
consider the the right shift of $(\ztb,\omb,1)$ by $e^{t\omb_1}$, which in the local  chart reads
$$
\left(
\begin{array}{ccc}
1&0 \\
\omb+\frac{\omb_1 t}{1+(\ztb,\omb_1)t}&1_n
\end{array}
\right)
\left(
\begin{array}{ccc}
1&\ztb(1+(\ztb,\omb_1)t) \\
0&1_n
\end{array}
\right)
\left(
\begin{array}{ccc}
1+(\ztb,\omb_1)t& 0\\
0&\frac{1}{1+(\ztb,\omb_1)t}1_n
\end{array}
\right).
$$
Assuming  $\varphi$ an $H$-invariant function, we find $(e^{t\omb_1 }\varphi)(\ztb,\omb)=
\varphi\bigl(\ztb+\ztb(\ztb,\omb_1)t, \omb+\frac{\omb_1 t}{1+(\ztb,\omb_1)t}\bigr)$, hence
\be
 y_i \varphi(\ztb,\omb)&=&\frac{\partial}{\partial {\om_i}} \varphi(\ztb,\omb)+\zt_i\sum_{k=1}^n \zt_k\frac{\partial}{\partial{\zt_k}} \varphi(\ztb,\omb),
\nn\\
x_i \varphi(\ztb,\omb)&=&\frac{\partial}{\partial{\zt_i}} \varphi(\ztb,\omb).\nn
\ee

A version of star-product on $\C P^n$ (regarded as a real manifold) was constructed in \cite{BBEW} as a homogeneous (delation-invariant) star product on $V=\C^{n+1}$. To compare it
with our result, consider its complexified version on $V\oplus V^*$:
\be
\phi*\psi=\phi\psi+\sum_{r=1}^\infty \Bigl(-\frac{t}{2\mu}\Bigr)^k\sum_{s=1}^r\sum_{k=1}^s\frac{(-1)^{r-k}k^{r-1}}{s!(s-k)!(k-1)!}
(z,w)^s\cdot \bigl(\frac{\partial}{\partial z}\tp \frac{\partial}{\partial w}\bigr)^s(\phi\tp \psi),
\label{bordemann}
\ee
where $\frac{\partial}{\partial z}\tp \frac{\partial}{\partial w}=\sum_{i=1}^{n+1}\frac{\partial}{\partial z_i}\tp \frac{\partial}{\partial w_i}$, and the dot means the classical multiplication.

The vector space $V\oplus V^*$ carries the natural representation of $GL(n+1)$,
which is extended by the group of delations $GL(1)$. The
multiplication (\ref{bordemann}) is invariant with respect the direct product $GL(n+1)\times GL(1)$.
In particular,  it restricts to homogeneous
functions of zero degree, regarded as functions on the projective space $\C P^{2n+1}$.

Let $\{e_i\}\subset V$ be the standard basis  and $\{e^i\}\subset V^*$ be its dual.
Consider the $GL(n+1)$-orbit $O$ passing through $o=e_1\oplus e^1$. The isotropy subgroup of this point
is $1\times GL(n)$, so this orbit is isomorphic to the coset space $GL(n+1)/1\times GL(n)$.
We extend the above parametrization of $GL(n+1)/GL(1)\times GL(n)$ to a
parametrization of $O$:
$$
(\ztb,\omb,a)\mapsto
\left(
\begin{array}{ccc}
1&0 \\
\omb&1
\end{array}
\right)
\left(
\begin{array}{ccc}
1&\ztb \\
0&1
\end{array}
\right)
\left(
\begin{array}{ccc}
a&0 \\
0&a^{-1}1_n
\end{array}
\right)o
=(a,a\omb)\oplus (a^{-1}+a^{-1}(\omb,\ztb),-a^{-1}\ztb).
$$
It can  be extended to a  local parametrization
$$
a=z_0,\quad
b=(\zb,\wb),\quad
\zt_i=-z_0 w_i,  \quad
\om_i=\frac{z_i}{z_0}, \quad i=1,\ldots,n.
$$
of $V\oplus V^*$ near $o$.
In this chart, the basic vector fields are represented as follows:
$$
\frac{\partial}{\partial z_0}\mapsto
\frac{\partial}{\partial a}+\frac{1}{a}\bigl(b+(\ztb,\omb)\bigr)\frac{\partial}{\partial b}
+\frac{1}{a}\sum_{i=1}^n (\zt_i\frac{\partial}{\partial \zt_i}-\om_i\frac{\partial}{\partial \om_i}),
\quad \frac{\partial}{\partial w_0}\mapsto a\frac{\partial}{\partial b},
$$
$$
\frac{\partial}{\partial z_i}\mapsto -\frac{\zt_i}{a}\frac{\partial}{\partial b}+\frac{1}{a}\frac{\partial}{\partial \om_i},
\quad
\frac{\partial}{\partial w_i}\mapsto \om_i a\frac{\partial}{\partial b}+a\frac{\partial}{\partial \zt_i}.
$$
In the new coordinates, a function $\phi$ on $V\oplus V^*$ is homogeneous if and only if
$$
\phi(\la a,\la^2 b, \la^2 \ztb, \omb)=\phi( a,b,\ztb, \omb).
$$
It is $GL(1)\times 1_n$-invariant if it is independent of $a$.
The correspondence between functions on $GL(n+1)/GL(1)\times GL(n)$ and homogeneous
$GL(1)\times 1_n$-invariant functions on $V\oplus V^*$ is
$\phi(\ztb, \omb)\mapsto \phi(b^{-1}\ztb, \omb)$, with the reverse correspondence being the specialization
at $b=1$.
We rewrite the action of the basic vector fields on such functions to find
$$
\frac{\partial}{\partial z_0}\mapsto
-\sum_{i=1}^n\bigr((\ztb,\omb)\zt_i\frac{\partial}{\partial \zt_i}+\om_i\frac{\partial}{\partial \om_i}\bigl),
\quad \frac{\partial}{\partial w_0}\mapsto \sum_{i=1}^n\zt_i\frac{\partial}{\partial \zt_i},
$$
$$
\frac{\partial}{\partial z_i}\mapsto \zt_i \sum_{k=1}^n\zt_k\frac{\partial}{\partial \zt_k}+\frac{\partial}{\partial \om_i},
\quad
\frac{\partial}{\partial w_i}\mapsto \om_i  \sum_{i=1}^n\zt_i\frac{\partial}{\partial \zt_i}+\frac{\partial}{\partial \zt_i}.
$$
Now it is easy to check the equality
$$
\sum_{i=1}^n\frac{\partial}{\partial z_i}\tp \frac{\partial}{\partial w_i}
=
\sum_{i=1}^n\frac{\partial}{\partial \om_i}\tp \frac{\partial}{\partial \zt_i}
+
\sum_{i,j=1}^n \zt_i\zt_j\frac{\partial}{\partial \zt_i}\tp \frac{\partial}{\partial \zt_j}
=\xb\tp \yb.
$$
The bidifferential operator from (\ref{bordemann}), when restricted to $(GL(1)\times 1_n)\times GL(1)$-invariant  functions, reads
\be
\id\tp \id +\sum_{r=1}^\infty \Bigl(-\frac{t}{2\mu}\Bigr)^r\sum_{s=1}^r\sum_{k=1}^s\frac{(-1)^{r-k}k^{r-1}}{s!(s-k)!(k-1)!}(\xb\tp \yb)^s.
\label{bordemann1}
\ee
 It is therefore a series in the same operator
$\xb\tp \yb$ as  (\ref{star_P^n_1}). This reduces comparison of the two star products to a comparison
of  power series in two variables, $t$ and $\xb\tp \yb$.
\begin{propn}
Operators (\ref{star_P^n_1}) and (\ref{bordemann1}) coincide upon the
identification $2\mu-t = \la$.
\end{propn}
\begin{proof}
The proof is based on the following formula
$$
\sum_{k_1+\ldots+ k_{m}=k} a_1^{k_1}a_2^{k_2}\ldots  a_{m}^{k_{m}}
=\sum_{i=1}^{m} \frac{a_i^{k+m-1}}{\prod_{j\not =i}{(a_i-a_j)}},
$$
which holds true for any communing variables $a_i$.
Applying this formula for $a_i=i$, we get
$$
\sum_{k_1+\ldots+ k_{m}=r-m} 1^{k_1}2^{k_2}\ldots  {m}^{k_{m}}
=\sum_{k=1}^{m} \frac{k^{r-1}}{\prod_{j<k}{(a_k-a_j)}\prod_{j>k}{(a_k-a_j)}}
=\sum_{k=1}^{m} \frac{(-1)^{m-k}k^{r-1}}{(k-1)!(m-k)!}.
$$
Put $\theta=\frac{t}{2\mu}$ and rearrange the series in (\ref{bordemann1})  as
$$
\id\tp \id+\sum_{r=1}^\infty \sum_{m=1}^r\Bigr(\sum_{k_1+\ldots+ k_{m}=r-m} (1\theta)^{k_1}(2\theta)^{k_2}\ldots  (m\theta)^{k_{m}}\Bigl)
\frac{(-\theta)^{m}(\xb\tp \yb)^{m}}{m!},
\nn
$$
and further as
$$
\id\tp \id+\sum_{m=1}^\infty  \Bigr(\sum_{r=m}^\infty \sum_{k_1+\ldots + k_{m}=r-m} (1\theta)^{k_1}(2\theta)^{k_2}\ldots  (m\theta)^{k_{m}}\Bigl)
\frac{(-\theta)^{m}(\xb\tp \yb)^{m}}{m!}.
\nn
$$
The summation $ \sum_{r=m}^\infty\sum_{k_1+\ldots + k_{m}=r-m} $
is rearranged to summation $ \sum_{k_1,\ldots, k_{m}=0}^\infty  $.
It contracts the sum in the brackets to
$
\frac{1}{(1-\theta)\ldots (1-m\theta)}.$
This immediately implies the statement.
\end{proof}

\section*{Appendix}
Below we collect some useful auxiliary algebraic  material about the properties
of "commutator" $[x,y]_a=xy- a yx$ defined in any associative algebra for some scalar $a$.
Next is a sort of "Jacobi identity" for such quasi-commutators.
\begin{lemma}
For any three elements $x,y,z$ of an associative algebra and any three scalars $a,b,c$
\be
[x,[y,z]_a]_b=[[x,y]_c,z]_{\frac{ab}{c}}+c [y,[x,z]_{\frac{b}{c}}]_{\frac{a}{c}}.
\label{Jacobi}
\ee
\end{lemma}
\noindent
The proof of this statement is elementary. Next state a useful fact, which is a sort of  Serre relation
for "adjacent root vectors" of higher weights.
\begin{lemma}
\label{Great Auxiliary}
Suppose  some elements $y,z,x$ of an associative algebra satisfy the identities
\be
[y,[y,z]_b]_{b^{-1}}=0,\quad [x,[x,y]_a]_{a^{-1}}=0, \quad
[x,z]=0.
\label{auxiliary}
\ee
for some invertible scalars $b,a$.
Then
$
[[x,y]_{a},[[x,y]_{a},z]_{b}]_{b^{-1}}=0.
$
\end{lemma}
\begin{proof}
Put $A=(b+b^{-1})$ and  $B=(a+a^{-1})$. The relations (\ref{auxiliary}) imply the equalities
\be
0&=&B(xy)^2z-(yx)(xy)z-AB(xy)z(xy)+A(yx)z(xy)+Bz(xy)^2-z(yx)(xy),
\nn\\
0&=&B(yx)^2z-(yx)(xy)z-AB(yx)z(yx)+A(yx)z(xy)+Bz(yx)^2-z(yx)(xy).
\nn\ee
The first line is a result of multiplication of the left identity (\ref{auxiliary}) by $x^2$ on the left and  using
the second and third identities  (\ref{auxiliary}). The second line is produced in a similar way, multiplying the left
equality by $x^2$ on the right.

Multiply the second line by  $a^2$ and add to the first line. The resulting equation will take the form
$$0=B[[x,y]_a,[[x,y]_a,z]_b]_{b^{-1}}, $$
as required.
\end{proof}
Remark that the hypothesis of the lemma is symmetric with respect to  replacement
of $a$ by $a^{-1}$, as well as $b$ by $b^{-1}$. Therefore, this replacement
can be made arbitrarily in the statement.

It follows that if pairs $x,y$ and $y,z$ commute as adjacent root vectors and $x,z$ as
distant root vectors, then the  $[x,y]_a$ commutes with $z$ as with $y$, as thought $z$ does not notice
the presence of $x$ in $[x,y]_a$.

\begin{lemma}
\label{lemma_xyz}
Suppose $x,y,z$ satisfy the relations
$$
[y,[y,x]_q]_{q^{-1}}=0,\quad [y,[y,z]_q]_{q^{-1}}=0,\quad
$$
and $x$ commutes with $z$.
Then
\be
[y,[x,[y,z]_q]_{q}]=0.
\label{xyyz}
\ee
\end{lemma}
\begin{proof}
Using the "Jacobi identity" \ref{Jacobi} with $a=c=q$, $b=1$ we get
$$
[y,[x,[y,z]_{q}]_{q}]=[[y,x]_q,[y,z]_{q}]+[x,[y,[y,z]_{q}]_{q^{-1}}.
$$
The right-hand side is zero: the first term vanishes as proved, the second due to the assumption.
This proves  (\ref{xyyz}).
\end{proof}
Remark that $q$ can be replaced by $q^{-1}$ in the two assumption equalities arbitrarily, as the double commutator
$[y,[y,x]_q]_{q^{-1}}$ is stable under this transformation.

\begin{proof}[Proof of Proposition \ref{prop_rel2}]
Let $\gm=\al+\mu$, where $\al$ is a simple positive root.
\be
[e_{\gm},f_{\gm}]&=&[[e_{\al},e_{\mu}]_q,[f_{\mu},f_{\al}]_{\bar q}]
=[[[e_{\al},e_{\mu}]_q,f_{\mu}],f_{\al}]_{\bar q}+[f_{\mu},[[e_{\al},e_{\mu}]_q,f_{\al}]]_{\bar q}
\nn\\[1pt]
&=&
[[e_{\al},[e_{\mu},f_{\mu}]]_q,f_{\al}]_{\bar q}+[f_{\mu},[[e_{\al},f_{\al}],e_{\mu}]_q]_{\bar q}
\nn\\[1pt]
&=&
[[e_{\al},\frac{q^{h_\mu}-q^{-h_\mu}}{q-q^{-1}}]_q,f_{\al}]_{\bar q}+[f_{\mu},[\frac{q^{h_\al}-q^{-h_\al}}{q-q^{-1}},e_{\mu}]_q]_{\bar q}
\nn\\[1pt]
&=&
\frac{-q^{-1}(1-q^2)}{q-q^{-1}}[e_{\al},f_{\al}]q^{-h_\mu}-\frac{q-q^{-1}}{q-q^{-1}}[f_{\mu},e_{\mu}]q^{h_\al}
\nn\\[1pt]
&=&
\frac{q^{h_\al}-q^{-h_\al}}{q-q^{-1}}q^{-h_\mu}+\frac{q^{h_\mu}-q^{-h_\mu}}{q-q^{-1}}q^{h_\al}=
\frac{q^{h_\al+h_\mu}-q^{-h_\al-h_\mu}}{q-q^{-1}},
\nn
\ee
\be
[e_{\mu},f_{\mu+\gm}]&=&[e_{\mu},[f_{\gm},f_{\mu}]_{\bar q}]=[f_{\gm},\frac{q^{h_\mu}-q^{-h_\mu}}{q-q^{-1}}]_{\bar q}=-\frac{1-q^{-2}}{q-q^{-1}}f_{\gm}q^{-h_\mu}=-q^{-1}f_{\gm}q^{-h_\mu},
\nn
\ee
\be
[e_{\gm},f_{\mu+\gm}]&=&[e_{\gm},[f_{\gm},f_{\mu}]_{\bar q}]=[\frac{q^{h_\gm}-q^{-h_\gm}}{q-q^{-1}},f_{\mu}]_{\bar q}=
\frac{q-q^{-1}}{q-q^{-1}}f_{\mu}q^{h_\gm}=f_{\mu}q^{h_\gm},
\nn
\ee
\be
[e_{\gm},f_{\mu+\gm+\nu}]&=&[e_{\gm},[f_{\gm +\nu},f_{\mu}]_{\bar q}]=-q^{-1}[f_{\nu}q^{-h_\gm},f_{\mu}]_{\bar q}
=
-q^{-1}(f_{\nu}q^{-h_\gm}f_{\mu} -\bar q f_{\mu}f_{\nu}q^{-h_\gm})=
\nn\\[1pt]
&=&-q^{-1}(\bar q f_{\nu}f_{\mu}q^{-h_\gm} -\bar q f_{\mu}f_{\nu}q^{-h_\gm})
=0,
\nn
\ee
\be
[e_{\mu+\gm},f_{\gm+\nu}]&=&[[e_\mu,e_{\gm}]_{q},[f_{\nu},f_{\gm}]_{\bar q}]=
[f_{\nu},[[e_\mu,e_{\gm}]_{q},f_{\gm}]]_{\bar q}
=[f_{\nu},[e_\mu,[e_{\gm},f_{\gm}]]_{q}]_{\bar q}
\nn\\[1pt]
&=&
[f_{\nu},[e_\mu,\frac{q^{h_\gm}-q^{-h_\gm}}{q-q^{-1}}]_{q}]_{\bar q}
=-\frac{1-q^{2}}{q-q^{-1}}(1-\bar q^{2}) f_{\nu}e_\mu q^{-h_\gm}
\nn\\[1pt]
&=&
(q-q^{-1}) f_{\nu}e_\mu q^{-h_\gm},
\nn
\ee
\end{proof}
\begin{proof}[Proof of Corollary \ref{prop_rel3}]
\be
[e_{\mu},f^k_{\mu+\gm}]&=&-q^{-1}f_{\mu+\gm}^{k-1}f_{\gm}q^{-h_\mu}-q^{-1}f_{\mu+\gm}^{k-2}f_{\gm}q^{-h_\mu}f_{\mu+\gm}+\ldots
\nn\\&=&
-q^{-1}f_{\mu+\gm}^{k-1}f_{\gm}q^{-h_\mu}-q^{-1+1}f_{\mu+\gm}^{k-1}f_{\gm}q^{-h_\mu}+\ldots=-q^{-1}\frac{q^{2k}-1}{q^{2}-1}f_{\mu+\gm}^{k-1}f_{\gm}q^{-h_\mu},
\nn
\\[1pt]
[e_{\gm},f_{\mu+\gm}^k]&=&f_{\mu}q^{h_\gm}f_{\mu+\gm}^{k-1}+f_{\mu+\gm}f_{\mu}q^{h_\gm}f_{\mu+\gm}^{k-2}+\ldots
\nn\\&=&
q^{-k+1}f_{\mu}f_{\mu+\gm}^{k-1}q^{h_\gm}+q^{-k+3}f_{\mu}f_{\mu+\gm}^{k-1}q^{h_\gm}+\ldots
=
q^{-k+1}\frac{q^{2k}-1}{q^{2}-1}f_{\mu}f_{\mu+\gm}^{k-1}q^{h_\gm},\nn
\\[1pt]
[e_{\gm},f^k_{\gm}]&=&f^{k-1}_{\gm}\frac{q^{h_{\gm}}-q^{-h_{\gm}}}{q-q^{-1}}+f^{k-2}_{\gm}\frac{q^{h_{\gm}}-q^{-h_{\gm}}}{q-q^{-1}}f_{\gm}+\ldots
\nn\\&=&
f^{k-1}_{\gm}\Bigl(\frac{q^{h_\gm}-q^{-h_\gm}}{q-q^{-1}}+\frac{q^{h_\gm-2}-q^{-h_\gm+2}}{q-q^{-1}}+\ldots
\nn\\&=&
f^{k-1}_{\gm}\Bigl( q^{h_\gm+1}\frac{1-q^{-2k}}{(q-q^{-1})^2}+q^{-h_\gm-1}\frac{1-q^{2k}}{(q-q^{-1})^2}\Bigr)
.\nn
\ee
\end{proof}
\vspace{20pt}
{\bf Acknowledgements}.
This research is supported in part by the RFBR grant 09-01-00504.
The author is grateful to Max-Plank institute for hospitality.
The article has been improved owing to valuable comments
of the referee, to whom the author is much indebted for his/her effort.

\section*{Erratum to the journal version}
\begin{enumerate}
\item
Definition of $S^\la$ before Proposition 3.3 should be as

{\em Denote by $S^\la=\la(S_-^{(1)})S_-^{(2)}\tp S_+^{(1)}\la^*(S_+^{(2)})$,
where $S_-\tp S_+$ is  a lift of $S_{\C^\la, \C^{\la^*}}\in M_{\la}^+\tp M_{\la^*}^-$ to $\U_\hbar(\p^-)\tp \U_\hbar(\p^+)$.}

It turns to the journal version if the lift is appropriate.
\item
Page 9, the ordering on roots (after definition of $e_\mu$ and $\tilde e_\mu$):

{\em The roots can be written in an orthogonal basis $\{\ve_i\}_{i=1}^n$ of weights of the 
natural representation as $\ve_i-\ve_j$, $i,j=1,\ldots,n$, $i\not=j$.
The lexicographical ordering on pairs $(i,j)$ induce
an ordering  on positive roots $\ve_i-\ve_j$, $i<j$,
consistent with the ordered basis $(\al_1,\al_2,\ldots, \al_{n-1})\subset \h^*$.}

\item
Formula (5.14): the scalar factor should be   $q^{(\eta_i,\la)}$ rather than $q^{-(\eta_i,\la)}$.
\item In Theorem 6.3 and Corollary 6.4: $x$ and $y$ should be interchanged. 
Another way to fix this error is to understand by $\cdot_\hbar$ the opposite multiplication 
in the RTT dual in Corollary 6.3; then $\U_\hbar$ should be taken with the opposite comultiplication.

The star product of Corollary 6.5 is correct
 because the classical multiplication in $\C[G]$ is commutative.
\end{enumerate}
\end{document}